\documentclass[dvips]{imsart}

\RequirePackage[OT1]{fontenc}
\usepackage{bbm}
\usepackage[numbers]{natbib}
\usepackage{amsthm,amsmath}
\usepackage{graphics,ifthen}
\usepackage{latexsym}
\usepackage{mathrsfs}
\usepackage{amssymb}

\startlocaldefs

\newtheoremstyle{example}{\topsep}{\topsep}%
     {}
     {}
     {\rmfamily}
     {}
     {\newline}
     {\thmname{#1}\thmnumber{ #2}\thmnote{ #3}}

   \theoremstyle{example}

\numberwithin{equation}{section}
\theoremstyle{plain}
\newtheorem{thm}{Theorem}[section]

\newtheorem{prop}{Proposition}[section]
\newtheorem{lem}{Lemma}[section]
\newtheorem{rem}{Remark}[section]

\newtheorem{cor}{Corollary}[section]

\newcommand{\Lower}[2]{\smash{\lower #1 \hbox{#2}}}
\newcommand{\ben}{\begin{enumerate}}
\newcommand{\een}{\end{enumerate}}
\newcommand{\bi}{\begin{itemize}}
\newcommand{\ei}{\end{itemize}}

\endlocaldefs

\begin{document}
\begin{frontmatter}
\title{Stick-breaking Pitman-Yor processes given the species sampling size\protect} \runtitle{conditional $\mathrm{GEM}(\alpha,\theta)$}

\begin{aug}
\author{\fnms{Lancelot F.} \snm{James}\thanksref{t1}\ead[label=e1]{lancelot@ust.hk}}

\thankstext{t1}{Supported in
part by the grants RGC-GRF 16300217 of the HKSAR and RGC-HKUST 601712 of the HKSAR.}

\runauthor{Lancelot F. James}

\affiliation{Hong Kong University of Science and Technology}

\address[a]{Lancelot F. James\\ The Hong Kong University of Science and
Technology, \\Department of Information Systems, Business Statistics and Operations Management,\\
Clear Water Bay, Kowloon, Hong Kong.\\ \printead{e1}.}

\contributor{James, Lancelot F.}{Hong Kong University of Science
and Technology}

\end{aug}

\begin{abstract}
Random discrete distributions, say $F,$ known as species sampling models,  represent a rich class of models for classification and clustering, in Bayesian statistics and machine learning. They also arise in various areas of probability and its applications. Pitman\cite{PitmanPoissonMix}, within the species sampling context, shows that mixed Poisson processes may be interpreted  as the sample size up till a given time or in terms of waiting times of appearance of individuals to be classified. He notes connections to some recent work in the Bayesian statistic/machine learning literature, with some more classical results, and their interpretation within a species sampling context.
Armed with the interpretations and results in ~\cite{PitmanPoissonMix}, we let $F:=F_{\alpha,\theta},$ be a Pitman-Yor process for $\alpha\in (0,1),$ and $\theta>-\alpha,$ with $\alpha$-diversity equivalent in distribution to $S^{-\alpha}_{\alpha,\theta},$ where $S_{\alpha,0}:=S_{\alpha}$ is a stable random variable, with density $f_{\alpha}(t)$ and $\mathbb{E}[\mbox e^{-\lambda S_{\alpha}}]={\mbox e}^{-\lambda^{\alpha}},$ and let $(N_{S_{\alpha,\theta}}(\lambda),\lambda\ge 0)$ denote a mixed Poisson process with rate $S_{\alpha,\theta}.$ In this paper we derive explicit stick-breaking representations of $F_{\alpha,\theta}$ given $N_{S_{\alpha,\theta}}(\lambda)=m,$ for each fixed $m=0,1,2,3,\ldots.$ More precisely, if $(P_{\ell})\sim \mathrm{PD}(\alpha,\theta)$, denotes a ranked sequence following the two parameter Poisson-Dirichlet distribution, we obtain explicit representations of the sized biased permutation of $(P_{\ell})|N_{S_{\alpha,\theta}}(\lambda)=m.$ Due to distributional results we shall develop in a more general context, it suffices to consider the stable case $F_{\alpha,0}|N_{S_{\alpha}}(\lambda)=m.$ Notably, since 
$S_{\alpha}|N_{S_{\alpha}}(\lambda)=0,$ has density ${\mbox e}^{\lambda^{\alpha}}{\mbox e}^{-\lambda t}f_{\alpha}(t),$ it follows that $F_{\alpha,0}|N_{S_{\alpha}}(\lambda)=0,$ is equivalent in distribution to the popular normalized generalized gamma process. Hence, we obtain explicit stick-breaking representations for the generalized gamma class recovering a not well known result in the unpublished manuscript of James~\cite{JamesPGarxiv}.
\end{abstract}



\end{frontmatter}
\section{Introduction}
Consider random discrete distribution functions $F(y)=\sum_{k=1}^{\infty}P_{k}\mathbb{I}_{\{U_{k}\leq y\},}$ where we can assume the $(P_{\ell}):=(P_{\ell},\ell\ge 1)$  are placed in ranked order, and $(U_{\ell})$, independent of $(P_{\ell}),$ are a collection of iid $\mathrm{Uniform}[0,1]$ random variables. Formally we can say $(P_{\ell})\in \mathcal{P}_{\infty}=\{\mathbf{s}=(s_{1},s_{2},\ldots):s_{1}\ge s_{2}\ge\cdots\ge 0 {\mbox { and }} \sum_{i=1}^{\infty}s_{i}=1\},$ where $\mathcal{P}_{\infty}$ denotes the space of ranked mass partitions summing to $1.$ As discussed in~\cite{BerFrag,Kingman75,Pit06,PY97}, there has been considerable interest in the laws and interpretation of objects in  
$\mathcal{P}_{\infty},$ within probability, statistics and related areas. The two most notable laws are the Poisson-Dirichlet law, say $\mathrm{PD}(\theta):=\mathrm{PD}(0,\theta)$ with parameter $\theta>0,$ see~\cite{Kingman75,Ewens, Pit96},and its two-parameter extension,$\mathrm{PD}(\alpha,\theta)$ with parameters $0\leq \alpha<1$ and $\theta>-\alpha.$ 
The latter, for $0<\alpha<1,$ was developed, see~\cite{BPY,PY92,PPY92,Pit96,PY97,Pit06}, where $(P_{\ell})\sim \mathrm{PD}(\alpha,\theta)$ may be interpreted within the context of ranked lengths of excursions of Brownian motion and more general Bessel processes of dimension $2-2\alpha.$ 

Since the early work of Ferguson~\cite{Ferg73} on the Dirichlet process,$F(y),$ and by-products arising from sampling from such distributions,  have served as major components in Bayesian nonparametric statisics and machine learning applications. When $(P_{\ell})\sim \mathrm{PD}(0,\theta)$ $F(y):=F_{0,\theta}(y)$ is a Dirichlet process.  $F(y):=F_{\alpha,\theta}(y),$ with $(P_{\ell})\sim \mathrm{PD}(\alpha,\theta),$ is often referred to as a Pitman-Yor process. However, due to the generality intractability of $(P_{\ell})\sim\mathrm{PD}(\alpha,\theta),$ the $(P_{\ell})$ are rarely used directly in statistical applications. Instead, other remarkable properties of $F_{0,\theta}$ and $F_{\alpha,\theta},$ including generalized Chinese restaurant processes and tractable stick-breaking representations~(arising from the size biased sampling re-arrangement of $(P_{\ell})$), are employed,(see~\cite{Broderick2,IJ2001,IJ2003,Pit96,Sethuraman}). These can be described within the context of species sampling (sequential capturing and tagging animals appearing in an eco-system) or in effect sampling $(Y_{1},\ldots,Y_{n})$ conditionally iid from $F_{\alpha,\theta}.$ Our goal in this paper is to provide interpretation and derive explicit stick-breaking representations for  $F_{\alpha,\theta}$ conditioned on the number of animals that have appeared up to a time $\lambda,$ where this number is modeled by a mixed Poisson process.  
These interpretations are obtained from an unpublished mansucript of Pitman~\cite{PitmanPoissonMix}.
Details will be explained below. As a by-product we obtain an explicit stick-breaking representation, in terms of tractable random variables,  for the popular $F(y)$ following a \textit{normalized generalized gamma process,} see~\cite{LPD}, which equates to the case of $F_{\alpha,0}$ given that $m=0$ animals have been trapped up till some time $\lambda.$ This representaton in the normalized generalized gamma case, without the conditional  interpretation with respect to $F_{\alpha,0}$  was first given in the unpublished work of~\cite{JamesPGarxiv}. Subsequent stick-breaking representations for the normalized generalized gamma process, that is $m=0,$ for general $\alpha,$ appear in~\cite{FavStick1,FavStick2,LauStick}. See also \cite{Fav1} for a description of $\alpha=1/2,m=0$ corresponding to the Normalized Inverse Gaussian process. In Section~\ref{m5} of this paper we shall present results for $\alpha=1/2$ and general $m=0,1,2,\ldots,$ based on the results of \cite{AldousPit,PPY92} as discussed in \cite{Pit02,Pit06}.

\begin{rem}See \cite[Section 5,2]{Pit02} for the work of  McCloskey~\cite{McCloskey} on the (normalized) generalized gamma process within a species sampling context.
\end{rem}

We first recount some aspects of species sampling models, and random distribution functions based on~\cite{Pit96}, (see also \cite{BerFrag,IJ2001,IJ2003,PPY92,Pit02,Pit06,PY92,PY97}). Let $F_{0,\theta},$ for $\theta>0,$ denote a random discrete distribution following a Dirichlet process law as in Ferguson~\cite{Ferg73}. 
Pitman~\cite{Pit96}, starting from a Bayesian viewpoint of the Blackwell-MacQueen Polya urn and its relation to  conditionally iid sampling $(Y_{1},\ldots,Y_{n})$ from $F_{0,\theta}$, presents a broad view of Fisher's model for species sampling, whereby within an eco-system animals/individuals are trapped and tagged/classified, one by one in a sequential fashion over time. The first $n$ individuals are classified into $K_{n}\leq n$ species with iid unique tags $(Y^{*}_{1},\ldots,Y^{*}_{K_{n}})$ and membership, in order of appearance, denoted by the set ${\{C_{1},\ldots,C_{K_{n}}\}},$ with $C_{j}:={\{i:Y_{i}=Y^{*}_{j}\}}$ having random multiplicities $|C_{j}|:=N_{j,n},$ constituting a random partition of the integers $[n]:=\{1,\ldots,n\}.$  In the case of $F_{0,\theta},$ the law of ${\{C_{1},\ldots,C_{K_{n}}\}},$ corresponds to the one parameter Chinese restaurant process with law denoted as $\mathrm{CRP}(\theta):=\mathrm{CRP}(0,\theta).$ Furthermore the relative frequencies $(N_{j,n}/n)$ placed in ranked order converge to  $(P_{\ell})\in \mathcal{P}_{\infty}=\{\mathbf{s}=(s_{1},s_{2},\ldots):s_{1}\ge s_{2}\ge\cdots\ge 0 {\mbox { and }} \sum_{i=1}^{\infty}s_{i}=1\}$ having a Poisson-Dirichlet law, denoted as $(P_{\ell})\sim \mathrm{PD}(\theta).$ While $(P_{\ell})$ has a fairly intractable distributional form, the limit of the relative frequencies $(N_{j,n}/n)$ converge 
to the size biased re-arrangement of $(P_{\ell}),$ described as $(\tilde{P}_{\ell}=(1-W_{\ell})\prod_{j=1}^{\ell-1}W_{i},\ell\ge 1)$ where $(W_{i})$ are iid $\mathrm{Beta}(\theta,1),$ and $\tilde{P}_{1}=1-W_{1}$ is the first size biased pick from $(P_{\ell}).$ $(\tilde{P}_{\ell})$ is said to have a $\mathrm{GEM}(\theta):=\mathrm{GEM}(0,\theta)$ distribution~\cite{Ewens}, and leads to the tractable stick-breaking representation of $F_{0,\theta},$ in particular, for $(U_{\ell}), (\tilde{U}_{\ell})$ iid $\mathrm{Uniform}[0,1]$ independent of $(P_{\ell}),(\tilde{P}_{\ell})$,
\begin{equation}
F_{0,\theta}(y)=\sum_{k=1}^{\infty}P_{k}\mathbb{I}_{\{U_{k}\leq y\}}= \sum_{k=1}^{\infty}\tilde{P}_{k}\mathbb{I}_{\{\tilde{U}_{k}\leq y\}}
\label{stickDP}
\end{equation}
\cite{Pit96} shows that one may extend the framework above by replacing $F_{0,\theta}$ with general $F$ defined as in~(\ref{stickDP}) by placing more general distributions on $(P_{\ell})\in \mathcal{P}_{\infty} $ and hence $(\tilde{P}_{\ell}).$ Of particular interest to us, as highlighted in~\cite{Pit96}, is the case where $(P_{\ell})\sim \mathrm{PD}(\alpha,\theta),$  where $\mathrm{PD}(\alpha,\theta)$ denotes a two parameter  Poisson Dirichlet distribution, for 
$0\leq \alpha<1, \theta>-\alpha,$ with $\mathrm{PD}(0,\theta)=\mathrm{PD}(\theta),$ with origins and applications to for instance excursion thoeory, random tree and graph models as described in \cite{BerFrag,PPY92,PY97,Pit06}. Its size-biased permutation $(\tilde{P}_{\ell}):=(\tilde{P}_{\ell}=(1-W_{\ell})\prod_{j=1}^{\ell-1}W_{i},\ell\ge 1)$ is such that the $W_{i}$ are now independent $\mathrm{Beta}(\theta+i\alpha,1-\alpha),$ for $i=1,2,\ldots.$ One may express this distribution as  
$(\tilde{P}_{\ell})\sim \mathrm{GEM}(\alpha,\theta).$ Conversely $\mathrm{Rank}((\tilde{P}_{\ell}))\sim \mathrm{PD}(\alpha,\theta),$ where $\mathrm{Rank}(\cdot)$ denotes the ranked rearrangement of the sequence into $\mathcal{P}_{\infty}.$  

Remarkably, the two-parameter family $\mathrm{PD}(\alpha,\theta)$ is the only case where its size-biased permutation consists of independent $(W_{i})$,~\cite{Pit95,Pit96}.  The corresponding $F:=F_{\alpha,\theta},$ is now commonly referred to as the  Pitman-Yor process, as named in \cite{IJ2001}, and is regularly employed in complex applications in Bayesian statistics and machine learning with particular utility in cases involving power law behavior.
\begin{rem}In relation to the above statements, one can construct many $F,$ such that its mass probabilities are of the form $(1-W_{k})\prod_{i=1}^{k-1}W_{i},$ for $(W_{i})$ independent or iid random variables. However, they do not have the interpretation as the size biased permutation of their ranked counterpart. See \cite{Doksum,GnedinPitman2,JamesNTR} for the case of $F(y)=\int_{0}^{\infty}S(t-)
\Lambda(dt,y)$ derived from the masses of a Neutral to the Right $(S(t):t>0)$  process under a homogeneous subordinator. 
\end{rem}
\subsection{Species sampling waiting times and counts} 
Hidden from the species sampling schemes above, are random mechanisms describing the time of appearance of individuals, or the sample size up till a given time. Pitman\cite{PitmanPoissonMix} points out that the early work of Fisher\cite{Fisher}, McCloskey~\cite{McCloskey} and other authors, already provide a rich mixed Poisson process framework to describe such mechanisms in a species sampling context. 
\cite{PitmanPoissonMix}, offers a fresh exposition on mixed Poisson models within this context which points out links to this older work and more recent works in the literature, in particular the recent appearance of mixed Poisson models in Bayesian non-parametric statistics and machine learning.  This, for instance connects to recent work on frequency of frequency (FoF) distributions~\cite{ZhouFoF}, where the population size is modelled by mixed Poisson random variables, as well as gives an interpretation of conditioning on latent mixed gamma variables, $(T_{r})$ described below, as employed in~\cite{JLP2} and elsewhere, in terms of waiting/arrival times. 
 
The following descriptions, including section~\ref{PK}, can be read from~\cite{PitmanPoissonMix}. The distributional expressions below are based on elementary conditioning arguments.
Let $A:=\sum_{k=1}^{\infty}A_{k},$ denote a non-negative almost surely finite (random)sum, such that one may set $(P_{k}:=A_{k}/A)\in\mathcal{P}_{\infty}.$ For $r=1,2\ldots,$ let $G_{r}:=\sum_{j=1}^{r}\mathbf{e}_{j}:=G_{r-1}+\mathbf{e}_{r}$ denote increasing sums of independent standard exponential variables. For $A$ independent of the sequence $(G_{r}),$ define, for each $r,$ $T_{r}=G_{r}/A,$ whence $(T_{r})$ may be interpreted as the sequence of waiting times of a mixed Poisson process $(N_{A}(t) ;t\ge 0)$ defined as 
$$
N_{A}(t)=\sum_{r=1}^{\infty}\mathbb{I}_{\{T_{r}\leq t\}}.
$$
That is, $T_{r}=\inf{\{t:N_{A}(t)=r\}}$, for $r=1,2,\ldots$. Within the species sampling context, $A$ has the interpretation as the \textit{total species abundance}, with $(A_{k})$ corresponding to abundance of type. $N_{A}(t)$ is number of animals/individuals appearing up till time $t,$ and for each $n=1,2,3,\ldots$ $T_{n}$ represents the time of appearance/trapping of the $n$-th individual. 
There is the following description of the conditional distribution  of $A$ given $N_{A}(\lambda)=0,$ that is to say given that no individuals have been observed up till time $\lambda\ge 0,$ 
\begin{equation}
\mathbb{P}(A\in da|N_{A}(\lambda)=0)=\mathbb{P}(A\in da|T_{1}>\lambda)=
\frac{{\mbox e}^{-\lambda a}\mathbb{P}(A\in da)}{\mathbb{E}[{\mbox e}^{-\lambda A}]}.
\label{AconPoisson}
\end{equation}
where (\ref{AconPoisson}) corresponds to that of a random variable $A_{0}(\lambda)$ formed by exponential tilting the distribution of $A$
and for $m=1,2,\ldots,$ and $\lambda>0$
\begin{equation}
\mathbb{P}(A\in da|N_{A}(\lambda)=m)=\mathbb{P}(A\in da|T_{m}=\lambda)=
\frac{a^{m}{\mbox e}^{-\lambda a}\mathbb{P}(A\in da)}{\mathbb{E}[A^{m}{\mbox e}^{-\lambda A}]}.
\label{AconPoisson2}
\end{equation}
Denote  a random variable specified by~(\ref{AconPoisson}) and (\ref{AconPoisson2}) as $A_{m}(\lambda).$ Furthermore, for brevity set $\gamma^{[m]}_{\lambda}(da):=\mathbb{P}(A\in da|N_{A}(\lambda)=m),$ for $m=0,1,2,\ldots$

\subsection{Poisson Kingman distributions given $N_{A}(\lambda)=m$}\label{PK}
If we assume that $(A_{k})$ are the ranked jumps of a subordinator say $(A(y):y\ge 0),$ then one may set $A:=A(1),$ satisfying $\mathbb{E}[{\mbox e}^{-\lambda A}]={\mbox e}^{-\psi(\lambda)},$ where $\psi(\lambda)=\int_{0}^{\infty}
(1-{\mbox e}^{-\lambda s})\rho(s)ds ,$ and $\rho(s)$ is a L\'evy density which we further assume satisfies $\int_{0}^{\infty}\rho(s)ds=\infty.$ Then, as defined in~\cite[Definition 3]{Pit02}(see also \cite{Kingman75}), $(P_{\ell})$ has the law of a \textit{Poisson-Kingman distribution with Levy density} $\rho,$ denoted $\mathrm{PK}(\rho).$ Furthermore, the law of $(P_{\ell})|A=t,$ is denoted a $\mathrm{PK}(\rho|t),$ and for a probability distribution $\gamma(dt)$ on $(0,\infty),$
$$
(P_{\ell})\sim \mathrm{PK}(\rho,\gamma)=\int_{0}^{\infty}\mathrm{PK}(\rho|t)\gamma(dt).
$$
where $\mathrm{PK}(\rho,\gamma)$ is a said to be a \textit{Poisson-Kingman distribution with L\'evy density $\rho$ and mixing distribution $\gamma$.} Although one may create an infinite number of distributions over $\mathcal{P}_{\infty},$
it is not obvious how to interpret the meaning of $(P_{\ell})$ with respect to $\gamma.$ 
However, in the present setting, since the law of $N_{A}$ depends on $(P_{\ell})$ only through $A,$  the law of $(P_{\ell})|A=t,N_{A}=m$ is equivalent to that of $(P_{\ell})|A=t\sim \mathrm{PK}(\rho|t).$  It is evident that $(P_{\ell})|N_{A}(\lambda)=m,$ for each $m=0,1,2,\ldots$ have Poisson-Kingman distributions with mixing distributions $\gamma^{[m]}_{\lambda}.$ That is the distributions $\mathrm{PK}(\rho,\gamma^{[m]}_{\lambda}).$
The above species sampling framework then allows one to interpret these laws.

\begin{rem} (\ref{AconPoisson2}) indicates that $(P_{\ell})|T_{m}=\lambda,$ has the same law as $(P_{\ell})|N_{A}(\lambda)=m,$  for $m=1,2,\ldots.$ Furthermore $\mathrm{PK}(\rho,\gamma^{[0]}_{\lambda})=\mathrm{PK}(\rho_{\lambda}),$
for $\rho_{\lambda}(s)={\mbox e}^{-\lambda s}\rho(s).$ 
\end{rem} 
\begin{rem}General descriptions of the sized biased permutations of $(P_{\ell})\sim \mathrm{PK}(\rho,\gamma^{[m]}_{\lambda}),$ can be obtained by employing the results of~\cite{PPY92}. 
\end{rem}

\subsection{Distributional results for $-\infty<\varrho_{1}<\varrho_{2}<\infty$}
We now describe a randomization relationship between general $\varrho$ biased random variables, and hence corresponding $(P_{\ell}).$ The results shows that one needs to only consider the cases $m=0,1,2,\ldots,$ and a suitable randomization to recover all possible laws.  
\begin{prop}\label{transfer}
Suppose that for a real number $\varrho,$ $A_{\varrho}(\lambda)$ is a positive random variable with density 
$$
\mathbb{P}(A_{\varrho}(\lambda)\in da)=
\frac{a^{\varrho}{\mbox e}^{-\lambda a}\mathbb{P}(A\in da)}{\mathbb{E}[A^{\varrho}{\mbox e}^{-\lambda A}]}.
$$
and consider for $-\infty<\varrho_{1}<\varrho_{2}<\infty,$ the random variable $Y_{\varrho_{2},\varrho_{1}}(\lambda)\overset{d}=G_{\varrho_{2}-\varrho_{1}}/{A_{\varrho_{1}}(\lambda)}.$ Then for $-\infty<\varrho_{1}<\varrho_{2}<\infty,$
$$
A_{\varrho_{2}}(\lambda+Y_{\varrho_{2},\varrho_{1}}(\lambda))\overset{d}=A_{\varrho_{1}}(\lambda)
$$
\end{prop}
\begin{proof}The density of $Y_{\varrho_{2},\varrho_{1}}(\lambda)$ can be expressed as
$$
\mathbb{P}(Y_{\varrho_{2},\varrho_{1}}(\lambda)\in dy)/dy=\frac{y^{\varrho_{2}-\varrho_{1}-1}\mathbb{E}[A^{\varrho_{2}}{\mbox e}^{-(\lambda+y)A}]}{\Gamma(\varrho_{2}-\varrho_{1})\mathbb{E}[A^{\varrho_{1}}{\mbox e}^{-\lambda A}]}
$$
It follows that $\mathbb{P}(A_{\varrho_{1}}(\lambda)\in da)=\int_{0}^{\infty}\mathbb{P}(A_{\varrho_{2}}(\lambda+y)\in da)\mathbb{P}(Y_{\varrho_{2},\varrho_{1}}(\lambda)\in dy)$
\end{proof}
\section{The $\mathrm{PD}(\alpha,\theta)$ case given $N_{S_{\alpha,\theta}}(\lambda)$}
If $(A(y):y\ge 0)$ is a stable subordinator, such that $S_{\alpha}:=A(1):=A$ is a stable random variable with Laplace transform $\mathbb{E}[{\mbox e}^{-\lambda S_{\alpha}}]={\mbox e^{-\lambda^{\alpha}}},$ with density denoted as $f_{\alpha}(t),$ then $(P_{\ell}:=A_{\ell}/S_{\alpha},\ell\ge 1)\sim \mathrm{PD}(\alpha,0).$ The corresponding  L\'evy density is given by $\rho_{\alpha}(s)=\alpha s^{-\alpha-1}/\Gamma(1-\alpha),$ see~\cite[eq. (54)]{Pit02}.
Hence $\mathrm{PD}(\alpha,0):=\mathrm{PK}(\rho_{\alpha}).$ The law of $(P_{\ell})|S_{\alpha}=t$ is denoted as $\mathrm{PD}(\alpha|t):=\mathrm{PK}(\rho_{\alpha}|t),$ and $(P_{\ell})\sim \mathrm{PD}(\alpha,\theta)=\int_{0}^{\infty}\mathrm{PD}(\alpha|t)f_{\alpha,\theta}(t)dt,$ where $f_{\alpha,\theta}(t)=t^{-\theta}f_{\alpha}(t)/\mathbb{E}[S^{-\theta}_{\alpha}]$ corresponds to the density of the random variable $S_{\alpha,\theta}.$ It follows that for $(P_{\ell})\sim \mathrm{PD}(\alpha,\theta),$
the law of $(P_{\ell})|N_{S_{\alpha,\theta}}(\lambda)=m,$ can be expressed as 
\begin{equation}
\mathbb{P}^{[m-\theta]}_{\alpha}(\lambda)=\int_{0}^{\infty}\mathrm{PD}(\alpha|t)\mathbb{P}(S_{\alpha,\theta}\in dt|N_{S_{\alpha,\theta}}(\lambda)=m),
\end{equation}
where $\mathbb{P}(S_{\alpha,\theta}\in dt|N_{S_{\alpha,\theta}}(\lambda)=0)=\mathbb{P}(S_{\alpha,\theta}\in dt|\frac{G_{1}}{S_{\alpha,\theta}}>\lambda),$ and ,for $m=1,2,\ldots$ there is the equivalence,
$$\mathbb{P}(S_{\alpha,\theta}\in dt|N_{S_{\alpha,\theta}}(\lambda)=m)=\mathbb{P}(S_{\alpha,\theta}\in dt|\frac{G_{m}}{S_{\alpha,\theta}}=\lambda),$$
and otherwise for each $m=0,1,2,\ldots,$ has the density
\begin{equation}
f^{[m-\theta]}_{\alpha}(t|\lambda)=\frac{{\mbox e}^{-\lambda t}t^{m-\theta}f_{\alpha}(t)}{\mathbb{E}[S^{m-\theta}_{\alpha}{\mbox e}^{-\lambda S_
{\alpha}}]},
\label{condden}
\end{equation}
for $\theta>-\alpha.$ Set, $\varrho:=m-\theta$ which can be any real number, and let 
$S_{\alpha,\varrho}(\lambda)$ denote a random variable with density~(\ref{condden}). $S_{\alpha,\varrho}(\lambda)$ is well defined for all $\lambda\ge 0$ and $\varrho=-\theta<\alpha,$ with $S_{\alpha,-\theta}(0)=S_{\alpha,\theta},$ and otherwise defined for any $\varrho$ and $\lambda>0.$ A special case of Proposition~\ref{transfer} leads to the next result.
\begin{cor}\label{PDrecovery}Let for $\varrho_{2}>\varrho_{1}$, $Y_{\varrho_{2},\varrho_{1}}(\lambda)\overset{d}=G_{\varrho_{2}-\varrho_{1}}/S_{\alpha,\varrho_{1}}(\lambda)$ then 
$$\mathbb{E}[\mathbb{P}^{[\varrho_{2}]}_{\alpha}(\lambda+Y_{\varrho_{2},\varrho_{1}}(\lambda))]=
\mathbb{P}^{[\varrho_{1}]}_{\alpha}(\lambda).
$$
As special cases, set $\varrho_{1}=-\theta,$ for $\theta>-\alpha,$ then $Y_{\varrho_{2},-\theta}(0)\overset{d}=G_{\varrho_{2}+\theta}/S_{\alpha,\theta}$ and
$$\mathbb{E}[\mathbb{P}^{[\varrho_{2}]}_{\alpha}(Y_{\varrho_{2},-\theta}(0))]=
\mathrm{PD}(\alpha,\theta).
$$
Conversely if $(P_{\ell})\sim \mathrm{PD}(\alpha,\theta)$ then $(P_{\ell})|G_{\varrho_{2}+\theta}/S_{\alpha,\theta}=\lambda\sim \mathbb{P}^{[\varrho_{2}]}_{\alpha}(\lambda).$
\end{cor}
\subsection{The normalized generalized gamma case, and size biased mixing distributions}
Throughout, let $(\tau_{\alpha}(y):y\ge 0)$ denote a generalized gamma subordinator, such that the density of 
$S_{\alpha,0}(\lambda)\overset{d}=\tau_{\alpha}(\lambda^{\alpha})/\lambda$ is equivalent to $f^{[0]}_{\alpha}(t|\lambda)=
{\mbox e}^{\lambda^{\alpha}}{\mbox e}^{-\lambda t}f_{\alpha}(t),$ and a normalized generalized gamma process can be expressed as
\begin{equation}
\frac{\tau_{\alpha}(y\lambda^{\alpha})}{\tau_{\alpha}(\lambda^{\alpha})}\overset{d}=\sum_{k=1}^{\infty}P^{[0]}_{\ell}(\lambda)\mathbb{I}_{\{U_{k}\leq y\}}
\end{equation}
where $(P^{[0]}_{\ell}(\lambda))\sim\mathbb{P}^{[0]}_{\alpha}(\lambda).$ Furthermore setting $\rho_{\alpha,\lambda}(s)={\mbox e}^{-\lambda s}\rho_{\alpha}(s),$ there is the dual representation
$$
\mathbb{P}^{[0]}_{\alpha}(\lambda)=\int_{0}^{\infty}\mathrm{PD}(\alpha|t)f^{[0]}_{\alpha}(t|\lambda)dt:=\mathrm{PK}(\rho_{\alpha,\lambda}).
$$
The generalized gamma process and its normalized version appear in many applications. While $(\tau_{\alpha}(y):y\ge 0)$ is known to arise from the operation of exponentially tilting a stable subordinator, we see that within the species sampling context it arises from conditioning a $(P_{\ell})\sim\mathrm{PD}(\alpha,n)$ on the event $N_{S_{\alpha,n}}(\lambda)=n,$ for each $n=0,1,2,\ldots$
\begin{rem}Random variables having the distribution of $S_{\alpha,0}(\lambda)\overset{d}=\tau_{\alpha}(\lambda^{\alpha})/\lambda$ can be generated by the methods in \cite{DevroyeGen}.
\end{rem}
\subsection{Special properties of the $\mathbb{P}^{[m]}_{\alpha}(\lambda)$ distribution}
As in, \cite[pages 64-66]{Pit06}, see also\cite{Pit99}, let $K_{m}$ for $m=1,2,\ldots,$  denote the number of blocks in a random partition of $[m]=\{1,2,\ldots,m\}$ under a $\mathrm{PD}(\alpha,0)$ Chinese restaurant process scheme. Its probability mass function can be expressed as,
$$
\mathrm{P}_{\alpha}^{(m)}(k):=\mathbb{P}_{\alpha,0}(K_m = k)=\frac{\alpha^{k-1}\Gamma(k)}{\Gamma(m)} S_\alpha(m,k),
$$ 
with $S_\alpha(m,k) =  \frac{1}{\alpha^k k!} \sum_{j=1}^k (-1)^j \binom{k}{j} (-j\alpha)_m$ denoting the generalized Stirling number of the second kind. Hereafter set $\Omega_{0}(\lambda^{\alpha})=\lambda^{-\alpha}/\alpha,$ $\Omega_{1}(\lambda^{\alpha})=1$ and generally for $m=1,2,\ldots$
$$
\Omega_{m}(\lambda^{\alpha})=\Gamma(m)\times \sum_{\ell=1}^{m}\mathrm{P}^{(m)}_{\alpha,0}(\ell)
\frac{{(\lambda^{\alpha})}^{\ell-1}}{\Gamma(\ell)}.
$$

It is known through the relation between cumulants and moments of a generalized gamma variable that
\begin{equation}
\mathbb{E}[S^{m}_{\alpha}{\mbox e}^{-\lambda S_
{\alpha}}]=\alpha{\mbox e}^{-\lambda^{\alpha}}\lambda^{\alpha-m}\times {\Omega}_{m}(\lambda^{\alpha})
\label{momentrep}.
\end{equation}
with $\mathbb{E}[S_{\alpha}{\mbox e}^{-\lambda S_
{\alpha}}]=\alpha\lambda^{\alpha-1}{\mbox e}^{-\lambda^{\alpha}}.$ Then for $m=1,2,\ldots,$ there are the representations for the density of $S_{\alpha}|N_{S_{\alpha}}(\lambda)=m,$
\begin{equation}
{f}^{[m]}_{\alpha}(t|\lambda)=\frac{\lambda^{m-\alpha}{\mbox e}^{\lambda^{\alpha}}{\mbox e}^{-\lambda t}t^{m}f_{\alpha}(t)}{\alpha\Omega_{m}(\lambda^{\alpha})}=\frac{t^{m} \times (
{\mbox e}^{\lambda^{\alpha}}{\mbox e}^{-\lambda t}f_{\alpha}(t))}{\mathbb{E}\left[{\left(\frac{\tau_{\alpha}(\lambda^{\alpha})}{\lambda}\right)}^{m}\right]},
\label{condSU}
\end{equation}
which is the $m$-th size biased density of $\tau_{\alpha}(\lambda^{\alpha})/\lambda$. Furthermore there is the representation of the corresponding random variable
\begin{equation}
S_{\alpha,m}(\lambda)\overset{d}=\frac{\tau_{\alpha}\left(\lambda^{\alpha}+G_{\frac{m}{\alpha}-K_{m}(\lambda)}\right)}{\lambda}\overset{d}=\frac{\tau_{\alpha}(\lambda^{\alpha})+G_{{m}-K_{m}(\lambda)\alpha}}{\lambda}.
\label{Sm}
\end{equation}
where $K_{m}(\lambda)$ is the discrete random variable corresponding to the conditional distribution of $K_{m}|N_{S_{\alpha}}(\lambda)=m,$ which has probability mass function
\begin{equation}
\mathbb{P}_{\alpha,0}(K_{m}(\lambda)=k)=\frac{\mathrm{P}^{(m)}_{\alpha,0}(k)\frac{\lambda^{k\alpha}}{\Gamma(k)}}{\sum_{\ell=1}^{m}\mathrm{P}^{(m)}_{\alpha,0}(\ell)\frac{\lambda^{\ell\alpha}}{\Gamma(\ell)}}
\label{condKu}
\end{equation} 
We close with an expression for the density of $Y_{m,m-k}(\lambda)/\lambda$ which will play a role in the description of the size biased ordering of $(P_{\ell})\sim\mathbb{P}^{[m]}_{\alpha}(\lambda).$
\begin{lem}\label{Ymmk}
Consider $Y_{m,m-k}(\lambda)\overset{d}=G_{k}/S_{\alpha,m-k}(\lambda),$ satisfying $\mathbb{E}[\mathbb{P}^{[m]}_{\alpha}(\lambda+Y_{m,m-k}(\lambda))]=\mathbb{P}^{[m-k]}_{\alpha}(\lambda)$ for each $k=1,\ldots,m.$ Then $Y_{m,m-k}(\lambda)/\lambda$ has density 
\begin{equation}
\frac{\Omega_{m}(\lambda^{\alpha}{(1+y)}^{\alpha})}{\Gamma(k)\Omega_{m-k}(\lambda^{\alpha})}y^{k-1}{(1+y)}^{\alpha-m}{\mbox e}^{-[\lambda^{\alpha}{(1+y)}^{\alpha}-\lambda^{\alpha}]}.
\label{Ydensity}
\end{equation}
\end{lem}
\begin{rem}Except for parametrization, the law of $K_{m}(\lambda),$~(\ref{condKu}), appears in~\cite[eq. (15), p. 1629]{ZhouFoF}, although $\lambda$'s relation to time is not considered in that context. This is also a special case of the distribution $\mathbb{P}(n(\pi)=k|U_{m}=\lambda)$ as described in~\cite[p. 85]{JLP2}, where $U_{m}\overset{d}=T_{m}$ in the general $\mathrm{PK}(\rho)$ setting.
\end{rem}

In the next sections we will describe results for the size biased re-arrangement of $(P_{\ell})|N_{S_{\alpha}}=m,$
where $(P_{\ell})\sim \mathrm{PD}(\alpha,0).$ Applications of Corollary~\ref{PDrecovery} will allow for representations in the general cases. We focus on the case of $m=0,$ corresponding to a normalized generalized gamma process, in Section~\ref{gg}. The case of $m=1$ in Section~\ref{sizegg} and general $m$ in Section~\ref{genm}.

\section{Stick-breaking representations for $(P_{\ell})\sim \mathrm{PD}(\alpha,\theta)$ given $N_{S_{\alpha,\theta}}(\lambda)=m$}

Suppose that $(\tilde{P}_{\ell})\sim\mathrm{GEM}(\alpha,\theta),$ then $F_{\alpha,\theta}(y):=\sum_{k=1}^{\infty}\tilde{P}_{\ell}\mathbb{I}_{\{\tilde{U}_{\ell}\leq y\}}$ is a Pitman-Yor process with parameters $(\alpha,\theta).$ The result in \cite{PPY92,Pit96,PY97} lead to the following exact decomposition
\begin{equation}
F_{\alpha,\theta}(y)=W_{1}F_{\alpha,\theta+\alpha}(y)+(1-W_{1})\mathbb{I}_{\{\tilde{U}_{1}\leq y\}}
\label{gemdecomp}
\end{equation}
where $F_{\alpha,\theta+\alpha}(y)=\sum_{k=2}^{\infty}(\tilde{P}_{\ell}/W_{1})\mathbb{I}_{\{\tilde{U}_{\ell}\leq y\}},$
and $(\tilde{P}_{\ell}/W_{1},\ell\ge 2)\sim\mathrm{GEM}(\alpha,\theta+\alpha),$ independent of the first size biased pick $(1-W_{1})\sim \mathrm{Beta}(1-\alpha,\theta+\alpha).$ Furthermore $\mathrm{Rank}((\tilde{P}_{\ell}/W_{1},\ell\ge 2))\sim\mathrm{PD}(\alpha,\theta+\alpha).$ This operation of \textit{size biased deletion} and ranking the normalized components is described in \cite[Chap 6]{PY97}. Repeated application of this procedure leads to a collection of mass partitions with laws 
$(\mathrm{PD}(\alpha,\theta+(k-1)\alpha)),k\ge 1).$ It suffices to  focus on the $\mathrm{PD}(\alpha,0)$ case where $\theta=0.$ These facts, as established in~\cite{PPY92}, may be derived from the corresponding  sequence of random variables $(S_{\alpha,(k-1)\alpha},k\ge 1),$ forming a Markov Chain satisfying $W_{k}= S_{\alpha,k\alpha}/S_{\alpha,(k-1)\alpha}$ and where for any $n,$ 
$(W_{1},\ldots,W_{n})$ is independent of $S_{\alpha,n\alpha},$ with $S_{\alpha}:=S_{\alpha,0}=S_{\alpha,n\alpha}/\prod_{j=1}^{n}W_{i}.$ In other words the joint law of $(W_{1},\ldots,W_{n},S_{\alpha,n\alpha})$ is given by,
\begin{equation}
\left[\prod_{k=1}^{n}f_{B_{k}}(w_{k})\right]f_{\alpha,n\alpha}(s)ds
\label{JointWn}
\end{equation}
where $f_{B_{k}}$ denotes the density of a $\mathrm{Beta}(k\alpha,1-\alpha)$ variable, and 
$f_{\alpha,n\alpha}(s)=s^{-n\alpha}f_{\alpha}(s)/\mathbb{E}[S^{-n\alpha}_{\alpha}]$ is the density of $S_{\alpha,n\alpha}.$ 
Precisely (\ref{JointWn}) can be written as 
$$
\frac{\alpha^{n-1}\Gamma(n)}{[\Gamma(1-\alpha)]^{n}\Gamma(n\alpha)}\prod_{k=1}^{n}w^{k\alpha-1}_{k}{(1-w_{k})}^{-\alpha}\times \frac{\alpha\Gamma(n\alpha)}{\Gamma(n)}s^{-n\alpha}f_{\alpha}(s).
$$
\begin{rem}See \cite[Section 5]{JamesLamp} for interpretations of~(\ref{gemdecomp}) in connection with the results in 
\cite[Theorem 1.3.1]{PY92} and \cite[Theorem 3.8, Lemma 3.11]{PPY92}.
\end{rem}
\subsection{An expression for the density of $(W_{1},\ldots,W_{n}),S_{\alpha,n\alpha}|N_{S_{\alpha}}(\lambda)=m$}
\begin{lem}\label{Wdensity}
Suppose that $(P_{\ell})\sim \mathrm{PD}(\alpha,0)$ with corresponding $(W_{1},\ldots,W_{n},S_{\alpha,n\alpha})$ having joint density in~(\ref{JointWn}). Then for each $m=0,1,2,\ldots,$
\begin{enumerate}
\item[(i)] $(P_{\ell})|N_{S_{\alpha}}(\lambda)=m$ has law $\mathbb{P}^{[m]}_{\alpha}(\lambda)=\int_{0}^{\infty}\mathrm{PD}(\alpha|t)f^{[m]}_{\alpha}(t|\lambda)dt.$
\item[(ii)]The (conditional) joint density of $(W_{1},\ldots,W_{n},S_{\alpha,n\alpha})|N_{S_{\alpha}}(\lambda)=m,$ can be expressed as, for each $n\ge 1,$
\begin{equation}
q_{m}(s/\prod_{i=1}^{n}w_{i}|\lambda)\left[\prod_{k=1}^{n}f_{B_{k}}(w_{k})\right]f_{\alpha,n\alpha}(s)
\label{jointVGm}
\end{equation}
where $q_{m}(t|\lambda)f_{\alpha}(t)=f^{[m]}_{\alpha}(t|\lambda).$ That is, 
\begin{equation}
q_{m}(t|\lambda)=\frac{\lambda^{m-\alpha}t^{m}{\mbox e}^{-[\lambda t-\lambda^{\alpha}]}}
{\alpha\Omega_{m}(\lambda^{\alpha})}.
\label{mweight}
\end{equation}
\end{enumerate}
\end{lem}
\begin{proof}This is a consequence of the results in \cite{PPY92}, and the extension of those results to the Poisson Kingman distributions with mixing, in the $\alpha$-stable case, as described in \cite{Pit02}. Where the mixing distribution is 
$q_{m}(t|\lambda)f_{\alpha}(t)=f^{[m]}_{\alpha}(t|\lambda).$ Otherwise, the inhomogeneous Markovian structure of 
$S_{\alpha,k\alpha}|S_{\alpha,(k-1)\alpha}=s$ for $k=1,2,\ldots,$ as established in~\cite{PPY92}, leads to a description of $(W_{1},\ldots,W_{n},S_{\alpha,n\alpha})|S_{\alpha}=s,N_{S_{\alpha}}(\lambda)=m,$ which just depends on $S_{\alpha}=s.$ The result is realized by then mixing over  $f^{[m]}_{\alpha}(s|\lambda),$ which is the density of $S_{\alpha}|N_{S_{\alpha}}(\lambda)=m.$
\end{proof}
\subsection{The variables $(R_{1}(\lambda),\dots,R_{n}(\lambda),n \ge 2)$}
Throughout, let $(\mathbf{e}_{i})$ denote a sequence of independent $\mathrm{exponential}~(1)$ variables.
Define partial sums $\tilde{G}_{k}=\sum_{i=1}^{k}\mathbf{e}_{i},$ with $\tilde{G}_{0}=0.$
Define for $k=1,2,\dots,$
\begin{equation}
R_{k}(\lambda):={\left(\frac{\tilde{G}_{k-1}+\lambda^{\alpha}}{\tilde{G}_{k}+\lambda^{\alpha}}\right)}^{\frac{1}{\alpha}}
\label{Rrep}
\end{equation}
The next result provides an important component to the stick-breaking representation and analysis. The proof follows  from elementary conditioning arguments and is omitted.
\begin{lem}\label{Rdist}Let for $k\ge 1,$ $R_{k}:=R_{k}(\lambda)$ be random variables described in~(\ref{Rrep}), then
given any $\lambda>0$ 
\begin{enumerate}
\item[(i)]$R_{1}$ has density 
$$
f_{R_{1}}(r_{1}|\lambda)=\alpha\lambda^{\alpha}{\mbox e}^{\lambda^{\alpha}}{\mbox e}^{-\lambda^{\alpha}/r^{\alpha}_{1}}r^{-\alpha-1}_{1}.
$$
\item[(ii)]$R_{j}|R_{1}=r_{1},\ldots,R_{j-1}=r_{j-1}$ has density $f_{R_{1}}(r_{j}|\lambda/\prod_{i=1}^{j-1}r_{i})$
\item[(iii)]Hence $(R_{1},\dots,R_{n})$ has joint density,
$$
f_{R_{1},\ldots,R_{n}}(\mathbf{r}_{n})=\alpha^{n}\lambda^{n\alpha}{\mbox e}^{\lambda^{\alpha}}{\mbox e}^{-\lambda^{\alpha}/(\prod_{l=1}^{n}r_{l})^{\alpha}}\prod_{l=1}^{n}r^{-(n-l+1)\alpha-1}_{l}
$$
\end{enumerate}
Furthermore,
$
\lambda^{\alpha}/{\prod_{l=1}^{k}R^{\alpha}_{l}}=\tilde{G}_{k}+\lambda^{\alpha}.
$
\end{lem}

\subsection{Case $m=0,$ $\mathbb{P}^{[0]}_{\alpha}(\lambda),$the normalized generalized gamma case}\label{gg}
In the case where $(P_{\ell})|N_{S_{\alpha}}(\lambda)=0\sim \mathbb{P}^{[0]}_{\alpha}(\lambda),$ having the law of the normalized jumps of a generalized gamma process, the joint density of 
$(W_{1},\ldots,W_{n},S_{\alpha,n\alpha})|N_{S_{\alpha}}(\lambda)=0,$ can be expressed as,
\begin{equation}
\frac{{\mbox e}^{-\frac{\lambda s}{\prod_{i=1}^{n}w_{i}}}{\mbox e}^{\lambda^{\alpha}}\alpha^{n}}{[\Gamma(1-\alpha)]^{n}}\prod_{k=1}^{n}w^{k\alpha-1}_{k}{(1-w_{k})}^{-\alpha}\times s^{-n\alpha}f_{\alpha}(s).
\label{joint1}
\end{equation}
The key to obtaining a simple description of the $(W_{i})$ in terms of random variables is facilitated by the following identity.

\begin{prop}\label{integral1}
The integral
$$
\int_{0}^{w_{1}}\cdots\int_{0}^{w_{n}}
\lambda^{n\alpha}
{\mbox e}^{-\frac{\lambda s}{(\prod_{l=1}^{n}r_{l})}}\prod_{l=1}^{n}(w_{l}-r_{l})^{\alpha-1}r^{-(n-l+1)\alpha-1}_{l}dr_{n}\ldots dr_{1}
$$
is equal to
$$
{[\Gamma(\alpha)]}^{n}s^{-n\alpha}{\mbox e}^{-\frac{\lambda s}{(\prod_{l=1}^{n}w_{l})}}\prod_{l=1}^{n}w^{l\alpha-1}_{l}.
$$

\end{prop}
\begin{proof}
For positive quantities $(v,t,\lambda,w)$ there is the integral identity,
$$
\Gamma(\alpha)v^{\alpha}t^{-\alpha}w^{\alpha-1}{\mbox e}^{-\frac{\lambda t}{vw}}=
\lambda^{\alpha}\int_{0}^{w}{(w-r)}^{\alpha-1}{\mbox e}^{-\frac{\lambda t}{vr}}r^{-\alpha-1}dr,
$$
which can be obtained by the change of variables $r=1/y$ and then $s=rx-1$ leading to a gamma integral.
The result follows by repeated usage of this identity starting with $r_{n}$ and initially setting $v=\prod_{i=1}^{n-1}r_{i}.$ Then integrate with respect to $r_{n-1}$, setting $v=w_{n}\prod_{i=1}^{n-2}r_{i}$ and so on.
\end{proof}
\subsubsection{Descriptions of $(W_{1},\ldots,W_{n})|N_{S_{\alpha}}(\lambda)=0$}
\begin{lem}\label{jointWzero}The joint density of $(W_{1},\ldots,W_{n})|N_{S_{\alpha}}(\lambda)=0,$ can be expressed as
$$
\left[\prod_{k=1}^{n}\frac{{(1-w_{k})}^{\alpha-1}}{\Gamma(1-\alpha)\Gamma(\alpha)}\right]\int_{0}^{w_{1}}\cdots\int_{0}^{w_{n}}
\prod_{l=1}^{n}(w_{l}-r_{l})^{\alpha-1}f_{R_{1},\ldots,R_{n}}(\mathbf{r}_{n})dr_{n}\ldots dr_{1}
$$
\end{lem}
\begin{proof} Replace $s^{-n\alpha}{\mbox e}^{-\frac{\lambda s}{(\prod_{l=1}^{n}w_{l})}}\prod_{l=1}^{n}w^{l\alpha-1}_{l}$ in (\ref{joint1}) with the integral in Proposition~\ref{integral1}. Integrate with respect to $f_{\alpha}(s)$ and rearrange terms.

\end{proof}
\begin{thm}\label{Wcase00}Let $(P^{[0]}_{\ell}(\lambda))\sim\mathbb{P}^{[0]}_{\alpha}(\lambda),$ corresponding in distribution to the ranked masses in $\mathcal{P}_{\infty}$ of the normalized jumps of a generalized gamma subordinator . Then the joint distribution of $(\tilde{P}_{\ell}(\lambda)),$ the size biased rearrangement of $(P^{[0]}_{\ell}(\lambda)),$ can be expressed as 
$$
(\tilde{P}_{1}(\lambda),\tilde{P}_{2}(\lambda),\ldots)=(1-W_{1},(1-W_{2})W_{1},\ldots)
$$
that is $\tilde{P}_{\ell}(\lambda)=(1-W_{\ell})\prod_{i=1}^{\ell-1}W_{i},$ where the  $(W_{k})$ are dependent random variables represented as
\begin{equation}
W_{k}=1-\beta^{(k)}_{({1-\alpha},\alpha)}[1-R_{k}(\lambda)]
\label{weightszero}
\end{equation}
for $(\beta^{(k)}_{(1-\alpha,\alpha)})$ iid $\mathrm{Beta}(1-\alpha,\alpha)$ variables independent of the $R_{k}(\lambda)$ defined  in (\ref{Rrep}) and Lemma~\ref{Rdist}.
\end{thm}
\begin{proof}Augmenting the integral expression in Lemma~\ref{jointWzero} shows that the $W_{1},\ldots,W_{n}|R_{1}(\lambda)=r_{1},\ldots, R_{n}(\lambda)=r_{n},N_{S_{\alpha}}(\lambda)=0$ has joint density
$$
\left[\prod_{k=1}^{n}\frac{{(1-w_{k})}^{\alpha-1}}{\Gamma(1-\alpha)\Gamma(\alpha)}\right]
\prod_{l=1}^{n}(w_{l}-r_{l})^{\alpha-1}$$
for $r_{i}<w_{i}<1,$ $i=1,\ldots,n.$ Which leads to the description in~(\ref{weightszero}).
\end{proof}
\begin{cor}
For each fixed $k,$ $1-W_{k}$ has the distribution of the first size-biased pick from a
$\mathbb{E}[\mathbb{P}^{[0]}_{\alpha}({(\tilde{G}_{k-1}+\lambda^{\alpha})}^{1/\alpha})]$ mass partition
$$(P_{l,k-1}(\lambda),\ell\ge 1):=\mathrm{Rank}((\tilde{P}_{\ell+k-1}(\lambda)/\prod_{i=1}^{k-1}W_{i},\ell\ge 1)),$$ with $W_{0}:=1,$ and $(P_{\ell,0}(\lambda),\ell\ge 1)\overset{d}=(P^{[0]}_{\ell}(\lambda))\sim\mathbb{P}^{[0]}_{\alpha}(\lambda).$
\end{cor}
\begin{proof}This follows from the description of the $(W_{k},R_{k}(\lambda))$ in Theorem \ref{Wcase00}, which leads to
$(P_{\ell,1}(\lambda),\ell\ge 1)\sim \mathbb{E}[\mathbb{P}^{[0]}_{\alpha}(\lambda/R_{1}(\lambda))]$ and more generally
$(P_{\ell,k-1}(\lambda),\ell\ge 1)\sim \mathbb{E}[\mathbb{P}^{[0]}_{\alpha}(\lambda/\prod_{i=1}^{k-1}R_{i}(\lambda))].$
\end{proof}

\subsubsection{Recovering the Pitman-Yor- $\mathrm{PD}(\alpha,\theta)$ case for $\theta
\ge 0$}
We now show how to recover the stick-breaking result in the (Pitman-Yor) $\mathrm{PD}(\alpha,\theta)$ case for $\theta\ge 0,$ in terms of $(R_{k}).$
\begin{cor}When $\lambda^{\alpha}=G_{\theta/\alpha}$ for $\theta\ge0,$
$$
\mathbb{E}[\mathbb{P}^{[0]}_{\alpha}(G^{1/\alpha}_{\theta/\alpha})]=\mathrm{PD}(\alpha,\theta)
$$ 
The corresponding,  
$$
R_{k}(G^{1/\alpha}_{\theta/\alpha}):={\left(\frac{\tilde{G}_{k-1}+G_{\theta/\alpha}}{\tilde{G}_{k}+G_{\theta/\alpha}}\right)}^{\frac{1}{\alpha}}.
$$
are now collections of independent beta distributed variables where one can set,
$R_{k}(G^{1/\alpha}_{\theta/\alpha}):=\beta_{\theta+(k-1)\alpha,1}\sim \mathrm{Beta}(\theta+(k-1)\alpha,1).$
Hence
$$
1-W_{k}:=\beta^{(k)}_{1-\alpha,\alpha}[1-\beta_{\theta+(k-1)\alpha,1}]=\beta_{1-\alpha,\theta+k\alpha}
$$
and are independent. Recall that $\mathrm{PD}(\alpha,0)=\mathbb{P}^{[0]}_{\alpha}(0)$, in this case $R_{1}(0)=0,$ and $R_{k}(0)=\beta_{(k-1)\alpha,1},$ for $k\ge 2.$
\end{cor}
\begin{proof} Since $\mathbb{P}^{[0]}_{\alpha}(\lambda)$ corresponds to the generalized gamma case this can be seen as an immediate consequence of \cite[Proposition 21]{PY97}. However from Corollary~\ref{PDrecovery}, we can also use the fact that 
$Y_{0,-\theta}(0):=G_{\theta}/S_{\alpha,\theta}\overset{d}=G^{1/\alpha}_{\theta/\alpha}.$
\end{proof}

\subsection{Case $m=1,$  $(P^{[1]}_{\ell}(\lambda))\sim\mathbb{P}^{[1]}_{\alpha}(\lambda)$}\label{sizegg}
For $(P_{\ell})\sim \mathrm{PD}(\alpha,0)$ recall that $(P_{\ell})|N_{S_{\alpha}}(\lambda)=1$ or equivalently $(P_{\ell})|G_{1}/S_{\alpha}=\lambda,$ has the law of a mass partition $(P^{[1]}_{\ell}(\lambda))\sim\mathbb{P}^{[1]}_{\alpha}(\lambda),$ where
$$
\mathbb{P}^{[1]}_{\alpha}(\lambda)=\frac{\lambda^{1-\alpha}}{\alpha}{\mbox e}^{\lambda^{\alpha}}\int_{0}^{\infty}\mathrm{PD}(\alpha|t)t{\mbox e}^{-\lambda t}f_{\alpha}(t) dt
$$

Since it is known that $Y_{1,0}(0):=G_{1}/S_{\alpha}\overset{d}=\mathbf{e}^{1/\alpha}_{1},$ it follows that 
$(P^{[1]}_{\ell}(\mathbf{e}^{1/\alpha}_{1}))\sim \mathrm{PD}(\alpha,0).$ The next result shows how $\mathbb{P}^{[0]}_{\alpha}(\lambda)$ can be expressed in terms of $\mathbb{P}^{[1]}_{\alpha},$ in general

\begin{prop}\label{onetozero}Let $Y_{1,0}(\lambda)=\lambda G_{1}/\tau_{\alpha}(\lambda^{\alpha}),$ then
\begin{enumerate}
\item[(i)]$Y_{1,0}(\lambda)\overset{d}={(\lambda^{\alpha}+\mathbf{e}_{1})}^{1/\alpha}-\lambda.$
\item[(ii)]
$
\mathbb{E}[\mathbb{P}^{[1]}_{\alpha}({(\lambda^{\alpha}+\mathbf{e}_{1})}^{1/\alpha})]=\mathbb{P}^{[0]}_{\alpha}(\lambda).
$
\item[(iii)]$(P^{[1]}_{\ell}({(\lambda^{\alpha}+\mathbf{e}_{1})}^{1/\alpha}))\sim \mathbb{P}^{[0]}_{\alpha}(\lambda).$
\item[(iv)]Furthermore $\lambda/R_{1}(\lambda)\overset{d}={(\lambda^{\alpha}+\mathbf{e}_{1})}^{1/\alpha}$
\end{enumerate}
\end{prop}
\begin{proof}Using Laplace transforms, $\mathbb{P}(Y_{1,0}(\lambda)>y)={\mbox e}^{-[{(\lambda+y)}^{\alpha}-\lambda^{\alpha}]},$ yielding [(i)]. Statements [(ii)] and [(iii)] are equivalent and follow from Corollary~\ref{PDrecovery} since $\lambda+ Y_{1,0}(\lambda)\overset{d}= {(\lambda^{\alpha}+\mathbf{e}_{1})}^{1/\alpha}.$
\end{proof}

\begin{thm}\label{TheoremP1} If $(P^{[1]}_{\ell}(\lambda))\sim\mathbb{P}^{[1]}_{\alpha}(\lambda),$ then the sequence $(\tilde{P}_{\ell}(\lambda)),$ obtained by size-biased sampling from $(P^{[1]}_{\ell}(\lambda)),$
can be represented as $\tilde{P}_{k}(\lambda)=(1-{W}_{k})\prod_{l=1}^{k-1}W_{l},$ 
with the following properties;
\begin{enumerate}
\item[(i)]The  $(W_{k})$ are generally dependent random variables represented as
\begin{equation}
W_{k}:={[\beta^{(k)}_{1-\alpha,\alpha}[(1-R_{k}(\lambda))/R_{k}(\lambda)]+1]}^{-1}
 \label{weightsone}
\end{equation}
for $(\beta^{(k)}_{(1-\alpha,\alpha)})$ iid $\mathrm{Beta}(1-\alpha,\alpha)$ variables independent of the $R_{k}(\lambda)$ defined  in (\ref{Rrep}).
\item[(ii)] Hence given $(R_{1}=r_{1},R_{2}=r_{2},\ldots),$ the $(W_{k})$ are conditionally independent. In particular $W_{k}|R_{1}=r_{1},\ldots,R_{k}=r$ equates in distribution to $1/(\beta^{(k)}_{1-\alpha,\alpha}[(1-r)/r]+1)$ with density,
\begin{equation}
\frac{\Gamma(1)}{\Gamma(1-\alpha)\Gamma(\alpha)}{w^{-1}{(1-w)}^{-\alpha}{(w-r)}^{\alpha-1}r^{1-\alpha}},
\label{1overb}
\end{equation}
 $0<r<w<1.$
\end{enumerate}  
\end{thm}
\begin{proof}The result follows as an immediate consequence of the forthcoming Lemma~\ref{jointWm} in the case of $m=1.$
\end{proof}

\begin{cor}
For each fixed $k,$ $1-W_{k}$ has the distribution of the first size-biased pick from a
$\mathbb{E}[\mathbb{P}^{[1]}_{\alpha}({(\tilde{G}_{k-1}+\lambda^{\alpha})}^{1/\alpha})]$ mass partition

$$(P_{\ell,k-1}(
\lambda),\ell\ge 1):=\mathrm{Rank}(\tilde{P}_{\ell+k-1}(\lambda)/\prod_{i=1}^{k-1}W_{i}),$$ 

with $W_{0}:=1.$ Furthermore,
$$
\mathbb{E}[\mathbb{P}^{[1]}_{\alpha}({(\tilde{G}_{k-1}+\lambda^{\alpha})}^{1/\alpha})]:=\mathbb{E}[\mathbb{P}^{[0]}_{\alpha}({(\tilde{G}_{k-2}+\lambda^{\alpha})}^{1/\alpha})],
$$
for $k=2,3,\ldots.$ In particular $(P_{\ell,0},\ell\ge 1)\sim\mathbb{P}^{[1]}_{\alpha}(\lambda)$ and $(P_{\ell,1},\ell\ge 1)\sim\mathbb{P}^{[0]}_{\alpha}(\lambda).$
\end{cor}
\begin{proof} Proof follows from the structure of $(W_{k},R_{k}(\lambda))$ in Theorem~\ref{TheoremP1}. The relation to $\mathbb{P}^{[0]}(\lambda)$ follows from 
Porposition~\ref{onetozero}.
\end{proof}

We now show how to recover the stick-breaking result in the $\mathrm{PD}(\alpha,\theta)$ case for $\theta>-\alpha,$ in terms of $(R_{k}).$ The next result follows from Corollary~\ref{PDrecovery} using $Y_{1,-\theta}(0)=G_{1+\theta}/S_{\alpha,\theta}\overset{d}=G_{{\frac{\theta+\alpha}{\alpha}}}^{1/\alpha}$
\begin{cor}There are the following results for $\theta>-\alpha.$ 
\begin{enumerate}
\item[(i)]$(P^{[1]}_{\ell}(G_{{\frac{\theta+\alpha}{\alpha}}}^{1/\alpha}))\sim \mathrm{PD}(\alpha,\theta)$
\item[(ii)] The $(R_{k}(G_{{\frac{\theta+\alpha}{\alpha}}}^{1/\alpha}))$ are independent $\mathrm{Beta}(\theta+k\alpha,1)$ variables.
\item[(iii)]Hence,
$$
\beta^{(k-1)}_{1-\alpha,\alpha}\frac{1-R_{k}(G_{{\frac{\theta+\alpha}{\alpha}}}^{1/\alpha}))}{R_{k}(G_{{\frac{\theta+\alpha}{\alpha}}}^{1/\alpha}))}\overset{d}=\frac{\gamma_{1-\alpha}}{\gamma_{\theta+k\alpha}},
$$
a ratio of independent $\mathrm{Gamma}(1-\alpha,1)$ and $\mathrm{Gamma}(\theta+k\alpha,1)$ variables. Hence the variables $W_{k}$ in (\ref{weightsone}) are independent with distribution,
$$
1-W_{k}=\beta_{1-\alpha,\theta+k\alpha}.
$$
\end{enumerate}
\end{cor}
\begin{rem}One may compare the $R_{k}(G_{{\frac{\theta+\alpha}{\alpha}}}^{1/\alpha})),$ with variables denoted $R_{1},R_{2},\ldots$ in \cite{PY97}. See in particular \cite[Theorem 15]{PY97}.
\end{rem}
\begin{rem} $(P_{\ell,k-1}(
G_{{\frac{\theta+\alpha}{\alpha}}}^{1/\alpha}),\ell\ge 1)\sim\mathrm{PD}(\alpha,\theta+(k-1)\alpha)$ for $k=1,2,\ldots.$
This equates to the $\mathrm{PD}(\alpha,\theta)$ Markov Chain on the space of mass partitions by the action of\textit{ size biased deletion} as described in \cite[Chap 6]{PY97}.
\end{rem}

\section{Stick-breaking for $(P_{\ell}^{[m]}(\lambda))\sim \mathbb{P}^{[m]}_{\alpha}(\lambda)$}\label{genm}
We now discuss the case where $m=1,2,3,\ldots$ is a general positive integer, corresponding to $(P_{\ell})|N_{S_{\alpha}}(\lambda)=m,$ when $(P_{\ell})\sim \mathrm{PD}(\alpha,0).$ 

The joint density of 
$(W_{1},\ldots, W_{n},S_{\alpha,n\alpha})|N_{S_{\alpha}}(\lambda)=m,$ as described in Lemma~\ref{Wdensity}, can be expressed as 
\begin{equation}
\frac{\lambda^{m-\alpha}{\mbox e}^{-\frac{\lambda s}{\prod_{i=1}^{n}w_{i}}}{\mbox e}^{\lambda^{\alpha}}\alpha^{n}}{[\Gamma(1-\alpha)]^{n}\alpha \Omega_{m}(\lambda^{\alpha})}\prod_{k=1}^{n}w^{-m}_{k}w^{k\alpha-1}_{k}{(1-w_{k})}^{-\alpha}\times s^{m}s^{-n\alpha}f_{\alpha}(s).
\label{jointm}
\end{equation}
\subsection{Description of joint density of $(W_{1},\ldots,W_{n})|N_{S_{\alpha}}(\lambda)=m$}
\begin{lem}\label{jointWm}The joint density of $(W_{1},\ldots,W_{n})|N_{S_{\alpha}}(\lambda)=m$ can be expressed 
in terms of the following conditional joint density of variables $(W_{1},\ldots,W_{n},R_{1,m},\ldots,R_{n,m})$
$$
\frac{\Omega_{m}(\lambda^{\alpha}/\prod_{l=1}^{n}r^{\alpha}_{l})}{\Omega_{m}(\lambda^{\alpha})}\left[\prod_{i=1}^{n}\frac{w^{-m}_{i}{(1-w_{i})}^{-\alpha}{(w_{i}-r_{i})}^{\alpha-1}r^{m-\alpha}_{i}}{\Gamma(1-\alpha)\Gamma(\alpha)}\right]
f_{R_{1},\ldots,R_{n}}(\mathbf{r}_{n})
$$
for $0<r_{i}<w_{i}<1, i=1,\ldots,n.$ Specifically a description of the density of $(W_{1},\ldots,W_{n})|N_{S_{\alpha}}(\lambda)=m$
is obtained by integrating over $(r_{1},\ldots,r_{n}).$
\end{lem}
\begin{proof} Replace $s^{-n\alpha}{\mbox e}^{-\frac{\lambda s}{(\prod_{l=1}^{n}w_{l})}}\prod_{l=1}^{n}w^{l\alpha-1}_{l}$ in (\ref{jointm}) with the integral in Proposition~\ref{integral1}. Integrate with respect to $s^{m}f_{\alpha}(s)$
leading to the evaluation of 
$$
\mathbb{E}[S^{m}_{\alpha}{\mbox e}^{-\frac{\lambda}{\prod_{i=1}^{n}r_{i}}S_{\alpha}}]
=\alpha{\mbox e}^{-\frac{\lambda^{\alpha}}{(\prod_{l=1}^{n}r^{\alpha}_{l})}}\lambda^{\alpha-m}\prod_{i=1}^{n}r^{m-\alpha}_{i}\times \Omega_{m}(\lambda^{\alpha}/\prod_{l=1}^{n}r^{\alpha}_{l}).
$$
An augmentation and arrangement of terms leads to the indicated joint density.
\end{proof}
Setting $m=1,$ in Lemma~\ref{jointWm} verifies the results in Theorem~\ref{TheoremP1}.
Furthermore if  $X\overset{d}=1/(\beta^{(k)}_{1-\alpha,\alpha}[(1-r)/r]+1)$
\begin{equation}
\mathbb{E}[X^{-(m-1)}]=r^{-(m-1)}\mathrm{E}[{(\beta_{1-\alpha,\alpha}(1-r)+r)}^{m-1}]
\label{normalising}
\end{equation}
which is the same as 
$$
\int_{r}^{1}\frac{w^{-m}{(1-w)}^{-\alpha}{(w-r)}^{\alpha-1}r^{1-\alpha}}{\Gamma(1-\alpha)\Gamma(\alpha)}dw.
$$
This identifies the normalizing constants in Lemma~\ref{jointWm} leading to further descriptions of the 
$(W_{1},\ldots,W_{n})$  and $(R_{1,m},\ldots,R_{n,m}).$
\subsection{Description of the random vector $(R_{1,m}(\lambda),\ldots,R_{n,m}(\lambda))$}

Let $(R_{1,m}(\lambda),\ldots,R_{n,m}(\lambda)):=(R_{1,m},\ldots,R_{n,m})$ denote the joint vector of random variables appearing in Lemma~\ref{jointWm}. It then follows from~(\ref{normalising}), that the joint density can be expressed as
\begin{equation}
f_{R_{1},\ldots,R_{n}}(\mathbf{r}_{n})\frac{\Omega_{m}(\lambda^{\alpha}/\prod_{l=1}^{n}r^{\alpha}_{l})}{\Omega_{m}(\lambda^{\alpha})}
\prod_{i=1}^{n}\mathrm{E}[{(\beta_{1-\alpha,\alpha}(1-r_{i})+r_{i})}^{m-1}]
\label{jointr}
\end{equation}
where, $f_{R_{1},\ldots,R_{n}}(\mathbf{r}_{n})$ is as in Lemma~\ref{Rdist}
$$
f_{R_{1},\ldots,R_{n}}(\mathbf{r}_{n})=\alpha^{n}\lambda^{n\alpha}{\mbox e}^{\lambda^{\alpha}}{\mbox e}^{-\lambda^{\alpha}/(\prod_{l=1}^{n}r_{l})^{\alpha}}\prod_{l=1}^{n}r^{-(n-l+1)\alpha-1}_{l}.
$$
Furthermore, from~(\ref{jointr}), it follows that $R_{1,m}(\lambda)$ has marginal density
  $$
f_{m}(r_{1}|\lambda)=\frac{\Omega_{m}(\lambda^{\alpha}/r^{\alpha}_{1})}{\Omega_{m}(\lambda^{\alpha})}
\mathrm{E}[{(\beta_{1-\alpha,\alpha}(1-r_{1})+r_{1})}^{m-1}]
f_{R_{1}}(r_{1}|\lambda)
$$
and for each $\ell \ge 1,$ $R_{\ell,m}|R_{1,m}=r_{1},\ldots,R_{\ell-1,m}=r_{\ell-1},\lambda$ has density
$f_{m}(r_{\ell}|\lambda/\prod_{i=1}^{\ell-1}r_{i}).$

\begin{prop}\label{PropPGWm1}
Let $R_{1,m}(\lambda),$ denote a random variable taking values in $[0,1],$ with density,
$$
f_{m}(r|\lambda)=\frac{\Omega_{m}(\lambda^{\alpha}/r^{\alpha})}{\Omega_{m}(\lambda^{\alpha})}
\mathrm{E}[{(\beta_{1-\alpha,\alpha}(1-r)+r)}^{m-1}]
f_{R_{1}}(r|\lambda).
$$
\begin{enumerate}
\item[(i)]Then by augmentation there exists a discrete random variable $N^{(1)}_{m}(\lambda),$ such that given $R_{1,m}(\lambda)=r$
its conditional probability mass function $\mathbb{P}(N^{(1)}_{m}(\lambda)=k|R_{1,m}(\lambda)=r):=q_{m}(k|r),$ can be expressed as 
\begin{equation}
q_{m}(k|r)=\frac{\frac{\Gamma(k-\alpha)}{\Gamma(k)\Gamma(1-\alpha)}
{m-1\choose k-1}{(1-r)}^{k-1}r^{m-k}}{\mathrm{E}[{(\beta_{1-\alpha,\alpha}(1-r)+r)}^{m-1}]}
\label{mpmf}
\end{equation}
for $k=1,\ldots,m,$ not depending on $\lambda.$
\item[(ii)]For each $\lambda>0,$ the marginal distribution of $N^{(1)}_{m}(\lambda)$ can be expressed as
$$
q_{k,m}(\lambda):=\mathbb{P}(N^{(1)}_{m}(\lambda)=k)
=\frac{\frac{\Gamma(k-\alpha)}{\Gamma(1-\alpha)}
{m-1\choose k-1}\Omega_{m-k}(\lambda^{\alpha})\alpha\lambda^{\alpha}}{\Omega_{m}(\lambda^{\alpha})}
$$
for $k=1,\ldots,m$
\item[(iii)] $R_{1,m}(\lambda)|N^{(1)}_{m}(\lambda)=k$ has the density of $\lambda/(\lambda+Y_{m,m-k}(\lambda)),$ given by 
\begin{equation}
f_{m,k}(r|\lambda):=\frac{\Omega_{m}(\lambda^{\alpha}/r^{\alpha})}{\Gamma(k)\Omega_{m-k}(\lambda^{\alpha})\alpha\lambda^{\alpha}}{(1-r)}^{k-1}{r}^{m-k}f_{R_{1}}(r|\lambda).
\label{rkmdensity}
\end{equation}
\item[(iv)]That is, 
$
\lambda/R_{1,m}(\lambda)\overset{d}=\lambda+Y_{m,m-N^{(1)}_{m}(\lambda)}(\lambda).
$
\end{enumerate}
\end{prop}
\begin{proof} The description is a consequence of Lemma~\ref{jointWm},Lemma~\ref{Ymmk} and the density of $Y_{m,m-k}(\lambda)/\lambda$ in (\ref{Ydensity}). These can be used to verify calculations of integrals and normalizing constants above. The results then follow from standard augmentation arguments.
\end{proof}
Let $(N^{(j)}_{m}(x_{j}),x_{j}>0, j=1,2,\ldots),$ denote for fixed $(x_{j})$ a collection of independent random variables having the same marginal distribution as $N^{(j)}_{m}(x_{j})\overset{d}=N^{(1)}_{m}(x_{j})$ as described in Proposition~\ref{PropPGWm1}. In addition let $(Y^{(j)}_{m,\ell}(x_{j}),\ell=0,\ldots,m-1,j=1,2,\ldots)$ denote a collection of independent random variables such that 

$$
Y^{(j)}_{m,\ell}(x_{j})\overset{d}=Y_{m,\ell}(x_{j})\overset{d}=\frac{x_{j} G_{m-\ell}}{\tau_{\alpha}\left(x^{\alpha}_{j}+G_{\frac{\ell}{\alpha}-K_{\ell}(x_{j})}\right)}
$$

\begin{cor}Let $(R_{1,m}(\lambda),\ldots,R_{n,m}(\lambda))$ denote the joint vector with distribution described by (\ref{jointr}) or equivalently by $\prod_{\ell=1}^{n}f_{m}(r_{\ell}|\lambda/\prod_{i=1}^{\ell-1}r_{i}).$ Set $\lambda_{0}=\lambda,$ and $\lambda_{j-1}=\lambda/\prod_{i=1}^{j-1}R_{i,m}(\lambda)$ for $j=2,3,\ldots,$ then one may set 
$$
\lambda_{j-1}/R_{j}(\lambda)=\lambda_{j-1}+Y^{(j)}_{m,m-N^{(j)}_{m}(\lambda_{j-1})}(\lambda_{j-1})
$$
for $j=1,\ldots,n.$
\end{cor}

\subsection{Description of $(W_{1},\ldots,W_{n})|N_{S_{\alpha}}(\lambda)=m$}

Hereafter, for $j=1,2,\ldots,$ set 
$$
N^{(j)}_{m}(\lambda_{j-1})=N_{j,m}(\lambda).
$$

\begin{thm}\label{TheoremPGWm} If $(P^{[m]}_{\ell}(\lambda))\sim\mathbb{P}^{[m]}_{\alpha}(\lambda),$ then the sequence $(\tilde{P}_{k}(\lambda)),$ obtained by size-biased sampling from $(P^{[m]}_{\ell}(\lambda)),$
can be represented as $\tilde{P}_{k}(\lambda)=(1-{W}_{k})\prod_{l=1}^{k-1}W_{l},$ 
with the following properties;
\begin{enumerate}
\item[(i)] Given $(R_{1,m}(\lambda)=r_{1},R_{2,m}(\lambda)=r_{2},\ldots),$ the $(W_{\ell})$ are conditionally independent,
with respective densities $B_{m,\alpha}(w_{\ell}|r_{\ell}),$ for $\ell=1,\ldots$ expressible as 
$$
B_{m,\alpha}(w_{\ell}|r_{\ell})=\frac{1}{\Gamma(1-\alpha)\Gamma(\alpha)}\frac{w^{-m}_{\ell}{(1-w_{\ell})}^{-\alpha}{(w_{\ell}-r_{\ell})}^{\alpha-1}r^{m-\alpha}_{\ell}}{\mathbb{E}[{(\beta_{1-\alpha,\alpha}(1-r_{\ell})+r_{\ell})}^{m-1}]}
$$
for $r_{\ell}<w_{\ell}<1,$ not depending on $\lambda.$
\item[(ii)] The density $B_{m,\alpha}(w|r)$ can be expressed in terms of an m-component mixture and it follows that, given $(R_{1,m}(\lambda)=r,N_{1,m}=k)$ for $k=1,\ldots,m.$  
\begin{equation}
W_{1}\overset{d}={[\beta^{(1)}_{({k-\alpha},\alpha)}{(1-r)}/{r}+1]}^{-1}
  \label{weights3}
\end{equation}
with density, for $r<w<1,$
$$
\frac{\Gamma(k)}{\Gamma(k-\alpha)\Gamma(\alpha)}{w^{-k}{(1-w)}^{k-\alpha-1}{(w-r)}^{\alpha-1}r^{k-\alpha}{(1-r)}^{1-k}}.
$$
\item[(iii)] Hence, the $(W_{\ell})$ are generally dependent random variables represented as
\begin{equation}
W_{\ell}:={[\beta^{(\ell)}_{N_{\ell,m}(\lambda)-\alpha,\alpha}[(1-R_{\ell,m}(\lambda))/R_{\ell,m}(\lambda)]+1]}^{-1}
 \label{weights2}
\end{equation}
\end{enumerate}  
\end{thm}
\begin{proof}These results follow from Lemma~\ref{jointWm} and the descriptions in Proposition~\ref{PropPGWm1}.
\end{proof}
\begin{cor}\label{number}Consider $(N_{\ell,m}(\lambda), \ell\ge 1),$ as described in Theorem~\ref{TheoremPGWm} and (\ref{weights2}). Then  $(N_{\ell,m}(\lambda), \ell\ge 1)|((W_{\ell},R_{\ell,m}(\lambda)),\ell\ge 1)$ are conditionally independent such that each $N_{\ell,m}(\lambda)-1$ has a $\mathrm{Binomial}(m-1, (1-W_{\ell}))$ distribution.
\end{cor}

\subsection{Recovering $\mathrm{GEM}(\alpha,\theta)$}
The $\mathrm{GEM}(\alpha,\theta)$ distribution is recovered by mixing with respect to the distribution of 
$Y_{m,-\theta}(0)\overset{d}=G_{m+\theta}/S_{\alpha,\theta}\overset{d}=G^{1/\alpha}_{\theta/\alpha+K_{m}},$ which has density, for $y>0,$
\begin{equation}
\frac{\alpha\Omega_{m}(y^{\alpha})}{\Gamma(m+\theta)\mathbb{E}[S^{-\theta}_{\alpha}]}y^{\alpha+\theta-1}{\mbox e}^{-y^{\alpha}}.
\label{GEMm}
\end{equation}
It follows from Corollary~\ref{number} that $(N_{\ell,m}(G_{m+\theta}/S_{\alpha,\theta}), \ell\ge 1),$  are independent with respective probability mass functions for each $\ell\ge 1,$
$$
{m-1\choose k-1}\mathbb{E}[{(1-W_{\ell})}^{k-1}W_{\ell}^{m-k}]
$$
where $W_{\ell}~\sim \mathrm{Beta}(\theta+\ell\alpha,1-\alpha),$ see also~\cite{Pit95}. $(W_{\ell},\ell\ge 1)|(N_{\ell,m}(G_{m+\theta}/S_{\alpha,\theta}), \ell\ge 1),$ are independent such that $W_{\ell}|N_{\ell,m}(G_{m+\theta}/S_{\alpha,\theta})=k,$ has a $\mathrm{Beta}(\theta+\ell\alpha+m-k,k-\alpha)$ distribution. $(R_{\ell,m}(G_{m+\theta}/S_{\alpha,\theta}), \ell\ge 1)|(N_{\ell,m}(G_{m+\theta}/S_{\alpha,\theta}), \ell\ge 1)$ are independent such that  $R_{\ell,m}(G_{m+\theta}/S_{\alpha,\theta})|N_{\ell,m}(G_{m+\theta}/S_{\alpha,\theta})=k,$ has a $\mathrm{Beta}(\theta+\ell\alpha+m-k,k)$ distribution. 
\subsubsection{Distribution of $(R_{\ell,m},\ell\ge 1)$ in the $\mathrm{GEM}(\alpha,\theta)$ case}
It is interesting to highlight the distribution of
$(R_{\ell,m}(G_{m+\theta}/S_{\alpha,\theta}),\ell \ge 1),$ which we don't believe is well known even in this ideal setting.
\begin{cor}Let $(R_{1,m}(\lambda),\ldots,R_{n,m}(\lambda)),$ for each $n,$ be the variables with joint density in~
(\ref{jointr}). Then in the $\mathrm{GEM}(\alpha,\theta)$ where the $(W_{\ell})$ are independent $\mathrm{Beta}(\theta+\ell\alpha,1-\alpha),$ One has that the collection of variables $(R_{\ell,m}(G_{m+\theta}/S_{\alpha,\theta}),\ell \ge 1),$
are independent with respective distributions on $[0,1],$
$$
\frac{\alpha\Gamma(m+\theta+\ell \alpha)\mathbb{E}[S^{-(\theta+\ell\alpha)}_{\alpha}]}{\Gamma(m+\theta+(\ell-1)\alpha)\mathbb{E}[S^{-(\theta+(\ell-1)\alpha)}_{\alpha}]}r^{\theta+\ell\alpha-1}_{\ell}\mathbb{E}[{(\beta_{1-\alpha,\alpha}(1-r_{\ell})+r_{\ell})}^{m-1}]
$$
where $\mathbb{E}[S^{-\theta}_{\alpha}]=\Gamma(\theta/\alpha+1)/\Gamma(\theta+1).$
\end{cor}
\begin{proof}The is obtained by integrating~(\ref{jointr}) with respect to (\ref{GEMm}) and deducing the right normalization which is obscured by telescoping. 
\end{proof}
\subsection{Results for $\alpha=1/2, \mathbb{P}^{[m]}_{1/2}(\lambda)$}\label{m5}
When $\alpha=1/2,$ $(P_{\ell})\sim \mathrm{PD}(1/2,0),$  corresponds to the laws of the ranked lengths of Brownian motion. $S_{1/2}(1)=S_{1/2}\overset{d}=1/(4G_{1/2})$ and hence 
$S_{1/2,0}(\lambda)\overset{d}=\tau_{1/2}(\sqrt{\lambda})/\lambda$ has an Inverse Gaussian distribution, denoted $\mathrm{IG}(\frac{1}{2\lambda},\frac{1}{2})$ with explicit density ${\mbox e}^{\sqrt{\lambda}}{\mbox e}^{-\lambda t-1/(4t)}t^{-3/2}/(2\sqrt{\pi}).$ Furthermore $S_{1/2,m}(\lambda),$ in addition to its representation in~(\ref{Sm}), which allows it to be easily sampled, has, as can be read from \cite[Section 4]{Pit99}, a Generalized Inverse Gaussian distribution with density, satisfying for all $\theta>-1/2,$ where $S_{1/2,\theta}\overset{d}=1/(4G_{\theta+1/2}),$
\begin{equation}
f^{[m]}_{1/2}(t|\lambda)=\mathbb{P}(\frac{1}{4G_{\theta+1/2}}\in dt|4G_{\theta+1/2}G_{m+\theta}=\lambda)/dt.
\label{halfdensity}
\end{equation}
where 
\begin{equation}
Y_{m,-\theta}(0)\overset{d}=4G_{\theta+1/2}G_{m+\theta}\overset{d}=G^{2}_{2\theta+K_{m}},
\label{halfmap}
\end{equation}
and $K_{m}$ is the number of blocks in a partition of $[m]$ under a $\mathrm{PD}(1/2,\theta)$ distribution.

Following \cite[Section 8]{Pit02} and \cite[Section 4.5]{Pit06}, let $\sqrt{2}L_{1}\overset{d}=S^{-1/2}_{1/2},$ where $L_{1}$
denotes the local time of Brownian motion starting at $0$ up till time $1$, and let $B_{1}$ denote  Brownian motion at time 1, which has a standard Normal distribution. Then $(P_{\ell})|L_{1}=s$ has the same distribution as $(P_{\ell})|S_{1/2}=s^{-2}/2$ which is $\mathrm{PD}(1/2|s^{-2}/2).$ Although perhaps not well known in the statistical literature, the result of Aldous and Pitman~\cite[Corollary 5]{AldousPit} leads to an explicit description of the sequence which we denote as $(\hat{P}_{\ell}(s),\ell\ge 1),$ which is the size biased ordering of a  mass partition with distribution  
$\mathrm{PD}(1/2|s^{-2}/2).$
In particular the first size biased pick $\hat{P}_{1}(s),$ satisfies
$$
\hat{P}_{1}(s)\overset{d}=\frac{B^{2}_{1}}{B^{2}_{1}+s^{2}}
$$
and, for each $\ell\ge 1$,
\begin{equation}
\hat{P}_{\ell}(s)=\frac{s^{2}}{Q_{\ell-1}+s^{2}}-\frac{s^{2}}{Q_{\ell}+s^{2}},
\label{condstick}
\end{equation}
where $Q_{0}=0,$ and $Q_{\ell}=\sum_{i=1}^{\ell}X_{i}$ for $X_{i}$ independent with common distribution equivalent to $B^{2}_{1}.$ The representations in~(\ref{condstick}) lead easily to representations of the size biased ordering of any Poisson Kingman distribution with mixing derived from the case of a $1/2$-stable subordinator, see~\cite{Pit02}.  Our relevant result takes the form described below.

\begin{cor}\label{cor5}Let $(P_{\ell}(\lambda))\sim \mathbb{P}^{[m]}_{1/2}(\lambda)$, with corresponding $1/2$-diversity $\sqrt{2}L_{1,m}(\lambda)\overset{d}=S^{-1/2}_{1/2,m}(\lambda).$ Then for $m=0,1,2,\ldots$ its size-biased re-arrangement 
$(\tilde{P}_{\ell}(\lambda))$ is equivalent in (joint) distribution to
$$
(\hat{P}_{\ell}(L_{1,m}(\lambda)); \ell\ge 1),
$$
which is specified by~(\ref{condstick}). Furthermore, as indicated by~(\ref{halfdensity}) and (\ref{halfmap}),
 $(\hat{P}_{\ell}(L_{1,m}(G^{2}_{2\theta+K_{m}})))\sim \mathrm{GEM}(1/2,\theta)$.
\end{cor}

\begin{rem}The case of $m=0,$ in Corollary~\ref{cor5}, as well as our results in Theorem~\ref{Wcase00} with $\alpha=1/2$ and $m=0,$ corresponds to the Normalized Inverse Gaussian(NIG) case.  One may compare these with the representations first obtained in~\cite{Fav1}. 
\end{rem}

\end{document}